\documentclass{amsart}
\usepackage[utf8]{inputenc}

\usepackage{amsmath,amsthm,amsfonts,amssymb,verbatim,eucal, dsfont, enumerate, enumitem}
\usepackage{ mathrsfs }
\usepackage[all,line]{xy}
\usepackage{enumerate}
\allowdisplaybreaks
\numberwithin{equation}{section}

\newtheorem{thm}[equation]{Theorem} 
\newtheorem{prop}[equation]{Proposition}
\newtheorem{lemma}[equation]{Lemma} 
\newtheorem{cor}[equation]{Corollary}
\newtheorem{example}[equation]{Example}
\newtheorem{remark}[equation]{Remark}

\newtheorem{defn}[equation]{Definition}

\DeclareMathOperator{\Ext}{Ext}
 
\DeclareMathOperator{\Ima}{Im}

\DeclareMathOperator{\Span}{Span}

\newcommand{\e}{\epsilon}

\newcommand{\Z}{{\mathbb Z}}

\newcommand{\Hom}{\mbox{\rm Hom\,}}

\renewcommand{\ker}{\mbox{\rm Ker\,}}

\newcommand{\KK}{\mathbb{K}}

\newcommand{\SH}{\mathscr{H}}

\newcommand{\CC}{\mathbb{C}}  
\newcommand{\ZZ}{\mathbb{Z}}

\newcommand{\HH}{{\rm HH}}

\newcommand{\qdha}{quantum Drinfeld Hecke algebra}
\newcommand{\q}{\mathbf{q}}
\newcommand{\tqdha}{truncated quantum Drinfeld Hecke algebra}
\newcommand{\trunh}{\hat{\SH}_{\q,\kappa,t}}

\DeclareMathAlphabet{\mathpzc}{OT1}{pzc}{m}{it}

\DeclareMathOperator{\ima}{Im}
\DeclareMathOperator{\ch}{char}

\begin{document}
\begin{abstract}
We consider deformations of quantum exterior algebras extended by finite groups.  Among these deformations are a class of algebras which we call truncated quantum Drinfeld Hecke algebras in view of their relation to classical Drinfeld Hecke algebras.  We give the necessary and sufficient conditions for which these algebras occur, using Bergman's Diamond Lemma.  We compute the relevant Hochschild cohomology to make explicit the connection between Hochschild cohomology and truncated quantum Drinfeld Hecke algebras.  To demonstrate the variance of the allowed algebras, we compute both classical type examples and demonstrate an example that does not arise as a factor algebra of a quantum Drinfeld Hecke algebra.

\end{abstract}
\title[Truncated quantum Drinfeld Hecke algebras]{Truncated quantum Drinfeld Hecke algebras and Hochschild cohomology}

\author{L.\ Grimley}
\address{Lauren Grimley, Mathematics Department, Spring Hill College,
Mobile, AL 36608, USA}
\email{lgrimley@shc.edu}
\author{C.\ Uhl}
\address{Christine Uhl, Department of Mathematics, St.\ Bonaventure University,
St\ Bonaventure, NY 14778}
\email{cuhl@sbu.edu}
\maketitle


\section{Introduction}

Drinfeld Hecke algebras occur naturally as deformations of the skew group algebra $S(V)\rtimes~G$, the (semi-direct product) algebra formed by a finite group $G$ acting on the polynomial algebra $S(V)$ over a vector space $V$.  Drinfeld Hecke algebras appear in diverse areas of mathematics, including representation theory and orbifold theory.  For the development of (quantum) Drinfeld orbifold algebras and their relation to Hochschild cohomology see \cite{F-GK}, \cite{DOA}, and \cite{Shroff}.  In \cite{SW2}, Shepler and Witherspoon used Hochschild cohomology to characterize Drinfeld Hecke algebras.  The same authors further developed deformations of skew group algebras, of which Drinfeld Hecke algebras arise as a special case, in \cite{SW3}.  Allowing for the introduction of possible noncommutativity, Levandovskyy and Shepler defined a generalization of Drinfeld Hecke algebras in \cite{LS}.  These algebras were analogously defined as a deformation of the skew group algebra $S_{\q}(V) \rtimes G$, where $S_{\q}(V)$ is the quantum polynomial algebra over $V$.  The class of algebras $\SH_{\q, \kappa}$, a factor algebra of the tensor algebra of $V$ extended by $G$, which satisfy the Poincar\'{e}-Birkhoff-Witt (PBW) property, they name quantum Drinfeld Hecke algebras.  In their paper, Levandovskyy and Shepler used noncommutative Gr{\"o}bner bases theory to enumerate the necessary and sufficient conditions for $\SH_{\q, \kappa}$ to be a quantum Drinfeld Hecke algebra.  Naidu and Witherspoon used this criteria to classify quantum Drinfeld Hecke algebras in terms of Hochschild cohomology in \cite{NW}, generating examples of quantum Drinfeld Hecke algebras coming from actions of complex reflection groups.

In this article, we replace $S_{\q}(V)$ with the quantum exterior algebra, $\Lambda_{\q}(V)$, and explore deformations of $\Lambda_{\q}(V) \rtimes G$.  One such class of deformations arises as a factor algebra of the $\SH_{\q, \kappa}$ of \cite{LS}.  Mirroring the language of \cite{LS}, we call these algebras for which the PBW property is satisfied truncated quantum Drinfeld Hecke algebras.  In Theorem \ref{ifftheorem}, we record the necessary and sufficient conditions for an algebra to be a truncated quantum Drinfeld Hecke algebra, modifying the technique of \cite{LS} and using Bergman's \cite{Bergman} Diamond Lemma.  Our conditions are nearly identical to those of \cite[Theorem 7.6]{LS} but with an analogous version of condition (i) and the addition of conditions (v) and (vi) to respect the truncation on $\Lambda_{\q}(V)$.  Because of condition (i), our characterization allows for truncated quantum Drinfeld Hecke algebras that do not arise as a factor algebra of a quantum Drinfeld Hecke algebra as our choice of group actions is more flexible.  We include one such example at the end of Section 5.  Section 3 establishes the conditions on Hochschild cohomology given by truncated quantum Drinfeld Hecke algebras.  In Section 4, we compute the relevant elements in Hochschild cohomology, culminating in establishing the cohomological elements which produce truncated quantum Drinfeld Hecke algebras.  With this result, we further develop the connection between Hochschild cohomology and truncated quantum Drinfeld Hecke algebras, allowing for another perspective for computations.  Finally, in addition to the example mentioned previously, Section 5 contains examples of truncated quantum Drinfeld Hecke algebras arising from diagonal group actions, whose Hochschild cohomology was computed separately by the first author in \cite{Grimley}.

Let $\KK$ be a field of characteristic not 2 and, unless otherwise noted, $\otimes = \otimes_{\KK}$.


\section{Truncated Quantum Drinfeld Hecke Algebras}

Let $\q=\{q_{ij} \in \KK^*\}$ be a set of nonzero elements of a field $\KK$ with $\ch(\KK) \neq 2$ and let $V$ be a $\KK$-vector space with basis $\{v_1, v_2, ..., v_n\}$.  Let $G \subset GL(V)$ be a finite group.  Restrict to $\KK$ with $\ch(\KK) \nmid |G|$.  We denote $g \in G$ acting on $v_j$ by $^g v_j$.   

\begin{defn}
Let $A$ be a $\KK$-algebra and assume a finite group $G$ acts on $A$.  The skew group algebra (also called cross product algebra) $A \rtimes G$ is $A \otimes \KK G$ as a vector space with multiplication $(ag)(bh)=a(^g b)gh$ for all $a,b \in A$ and $g,h \in G$.
\end{defn}

We suppress the $\otimes$ in $A \rtimes G$ whenever the context is clear.

Let $t$ be an indeterminate and $\kappa: V \times V \rightarrow \KK G$ be a bilinear function.   Define an associative $\KK$-algebra, our main algebra of study, $$\hat{\SH}_{\q,\kappa,t} := T(V) \rtimes G[t] / (v_jv_i - q_{ij}v_iv_j-\kappa (v_i,v_j)t, v_i^2).$$  
We may also denote $\kappa(v_i, v_j)=\sum_{g \in G} \kappa_g (v_i, v_j)g$, dividing $\kappa$ into its $g$-components.  One should compare $\hat{\SH}_{\q,\kappa,t}$ of this paper to $\SH_{\q,\kappa,t}$ of \cite{NW}, in which no truncation conditions occur, and $\SH_{\q,\kappa}$ of \cite{LS}, in which $t=1$ and truncation conditions are not included.

Assigning elements in $\KK G$ and $t$ degree 0 and $v_i$ degree 1, $\trunh$ is a filtered algebra.

\begin{defn}
We call $\hat{\SH}_{\q,\kappa,t}$ a truncated quantum Drinfeld Hecke algebra over $\KK[t]$ if it has a Poincar\'{e}-Birkhoff-Witt (PBW) basis over $\KK[t]$.

\end{defn}

In the special case that $t=1$, we call the algebras $\hat{\SH}_{\q,\kappa,1}$ for which there exists a PBW basis a \tqdha\ over $\KK$.  The $t=1$ case will be the main interest of this paper.  The inclusion of the indeterminate is necessary for comparing to conditions on deformations discussed in Sections 3 and 4.  

The goal of this Section is to determine precisely the conditions on $\kappa$ and $\q$ for which $\hat{\SH}_{\q,\kappa,1}$ is a \tqdha\ over $\KK$.  First we will need to develop some notation and definitions.  For each $g \in G$, let $^g v_j = \sum_{i=1}^n g_i^j v_i$ for some scalars $g_i^j \in \KK$.  As in Levandovskyy and Shepler ~\cite[Definition 3.1]{LS}, we make use of the quantum minor determinant.
\begin{defn}
The quantum $(i,j,k,l)$-minor determinant of a group element $g \in G$ is given by $det_{ijkl}(g):=g_k^i g_l^j - q_{ij}g_l^i g_k^j$.
\end{defn}

Let $$\Lambda_{\q}(V):= \KK \langle v_1, v_2, ..., v_n | v_j v_i = q_{ij} v_i v_j \textrm{ and } v_i^2 = 0 \textrm{ for }  i \neq j \in \{1, 2, ..., n\} \rangle$$ be the quantum exterior algebra on $V$.

\begin{lemma}[\cite{NW}, Lemma 4.2]\label{LambdaAut} 
$G$ acts as an automorphism on $\Lambda_{\q}(V)$ if and only if for all $g \in G$, 
\begin{center}
    $det_{srji}(g)=-q_{ij}det_{srij}(g) \textrm{ for all } 1 \leq s,r,i,j \leq n \textrm{ and } s \neq r, \,\, i \neq j,$ \\
    $g_i^rg_j^r(1+q_{ij})=0$ for all $i < j$.
\end{center}
\end{lemma}

We include a proof here for completeness.

\begin{proof} The proof is a generalization of \cite[Lemma 3.2]{LS}. 
Consider $${}^g(v_r){}^g(v_s) - q_{sr}{}^g(v_s){}^g(v_r) = \sum_{j,i} \textrm{det}_{srij}(g) v_jv_i$$ where $s \neq r$.  We can rewrite the right hand side as 
\begin{align*}
\sum_{j\leq i} \textrm{det}_{srij}(g) v_j v_i + \sum_{j > i} \textrm{det}_{srij}(g) v_jv_i &=\sum_{j < i} \textrm{det}_{srij}(g) v_jv_i + q_{ij}\sum_{j > i} \textrm{det}_{srij}(g) v_iv_j + \sum_{j} \textrm{det}_{srjj}(g) v_jv_j \\
&=\sum_{j < i} (\textrm{det}_{srij}(g) + q_{ji} \textrm{det}_{srji}(g)) v_jv_i + \sum_{j} \textrm{det}_{srjj}(g) v_jv_j.
\end{align*}

Since $v_jv_j = 0$ in $\Lambda_{\q}(V)$ and $\Lambda_{\q}(V)$ has a $\KK$-basis given by $\{v_1^{\alpha_1}v_2^{\alpha_2} \cdots v_n^{\alpha_n} | \alpha_i \in \{0,1\}\}$, the right hand side vanishes exactly when $$\textrm{det}_{srji}(g)=-q_{ij}\textrm{det}_{srij}(g) \textrm{ for all } 1 \leq s,r,i,j \leq n \textrm{ and } i \neq j, \,\,  s \neq r.$$
Similarly, we consider ${}^g(v_rv_r) - {}^g(v_r){}^g(v_r) = 0 - \sum_{i,j} g_i^rg_j^r v_iv_j$. Simplifying this expression, we get the second condition of the Lemma.
\end{proof} 

\begin{defn}[\cite{LS}, Definition 3.4] 
The parameter $\kappa$ is called a quantum 2-form if, for all $i$ and $j$, $\kappa(v_i,v_i)=0$ and $\kappa(v_j,v_i)=-q_{ij}^{-1}\kappa(v_i, v_j)$.
\end{defn}

\begin{prop}\label{forward}
If $\hat{\SH}_{q,\kappa,1}$ is a \tqdha, then \begin{enumerate}[label=(\roman*)] \item $\kappa$ is a quantum 2-form, 
\item $q_{ij} = q_{ji}^{-1}$ for $i \neq j$, and
\item $G$ acts on $\Lambda_{\q}(V)$ by automorphisms. \end{enumerate} 
\end{prop}

\begin{proof}
Suppose $\hat{\SH}_{\q,\kappa,1}$ is a \tqdha.  Then $\hat{\SH}_{\q,\kappa,1}$ has a PBW basis.  Since $v_i^2 = 0 = q_{ii} v_i^2 + \kappa(v_i,v_i)$, we have that $\kappa(v_i,v_i) = 0$.  When $i \neq j$, we have
\begin{align*}
v_jv_i & = q_{ij} v_iv_j + \kappa(v_i,v_j)\\
& = q_{ij}(q_{ji}v_jv_i + \kappa(v_j,v_i)) + \kappa(v_i,v_j)\\
& = q_{ij}q_{ji}v_jv_i + q_{ij}\kappa(v_j,v_i) + \kappa(v_i,v_j).
\end{align*}

Thus we have that $q_{ij} \neq 0$ for all $i \neq j$, as well as $q_{ij} = q_{ji}^{-1}$ for all $i \neq j$.  Furthermore $\kappa(v_j,v_i) = -q_{ij}^{-1}\kappa(v_i,v_j)$.  The two conclusions on the parameters is exactly conditions (i) and (ii).  

Now, for all $h \in G$ and $r \neq s$, we have
\begin{align*}
0 &= h (v_s v_r)h^{-1} - h(q_{rs}v_rv_s + \sum_{g \in G}\kappa_g(v_r,v_s)g)h^{-1}  \\
& = \,^h v_s \,^hv_r - q_{rs} \,^hv_r \,^hv_s - \sum_{g \in G} \kappa_g(v_r,v_s) hgh^{-1} \\
&=\sum_{i,j} (h_i^s h_j^r v_iv_j - q_{rs}h_i^rh_j^s v_iv_j)-\sum_{g \in G} \kappa_g(v_r,v_s) hgh^{-1}\\
& = \sum_{i,j} \textrm{det}_{rsji}(h) v_i v_j - \sum_{g \in G} \kappa_{h^{-1}gh}(v_r,v_s)g\\
& = \sum_{i < j}(\textrm{det}_{rsji}(h) + q_{ij} \textrm{det}_{rsij}(h))v_iv_j + \sum_i \textrm{det}_{rsii}(h) v_i^2 \\
& \hspace{30mm} + \sum_{g \in G}(\sum_{i <j} \textrm{det}_{rsij}(h) \kappa_g(v_i,v_j) - \kappa_{h^{-1}gh}(v_r,v_s))g.
\end{align*}

Since each term needs to be zero, we recover that $\det_{rsij}(g)=-q_{ji}\det_{rsji}(g)$ for $r \neq s$ and $i \neq j$.  Similarly we consider $0 = h(v_rv_r)h^{-1} - h(0)h^{-1}$ to recover $(1+q_{ij})(h_i^rh_j^r) = 0$ for all $i < j$. Therefore, by Lemma ~\ref{LambdaAut}, $G$ acts on $\Lambda_{\q}(V)$ by automorphisms.
\end{proof}

This result agrees with \cite[Proposition 3.5]{LS} except that conditions (ii) and (iii) are modified by the added truncation.  Notice, in particular, that we get no conditions on the diagonal terms $q_{ii}$.  Had we considered truncation at larger powers or no truncation on $v_i$, we would have required $q_{ii}=1$.

With all of this structure, we wish to determine the necessary and sufficient criteria for which $\hat{\SH}_{\q,\kappa,1}$ will be a \tqdha.  Recall that this means determining the conditions for which $\hat{\SH}_{q,\kappa,1}$ is a PBW algebra.  According to Bergman's Diamond Lemma \cite{Bergman}, it is enough to check a minimal set of words in an algebra to confirm that the algebra is PBW.  While Bergman says that the main results of his paper are trivial, but far from clear explicitly, it is the crux of our argument.  For completeness, we include the statement of Bergman's Diamond Lemma here.

\begin{thm}[\cite{Bergman}, Theorem 1.2]
Let $S$ be a reduction system for a free associative algebra $\KK \langle X \rangle$, and $\leq$ a semigroup partial ordering on $\langle X \rangle$, compatible with $S$, and having descending chain condition.  All ambiguities of $S$ are resolvable relative to $\leq$ if and only if all elements of $k \langle X \rangle $ are reduction-unique under $S$.
\end{thm}

In this setting $X=\{v_1, \cdots , v_n \} \cup  G$ and $S$ are the relations on $\hat{\SH}_{q,\kappa,1}$.  This result implies we must check the following ambiguities: $v_kv_jv_i$ for $k > j > i$, $v_jv_i^{2}$ for $j > i$ and $v_i^{2} =0$, $v_j^{2}v_i$ for $j > i$ and $v_j^{2}=0$, $gv_i^{2}$ for $v_i^{2} = 0$ and $g \in G$, $gv_jv_i$ for $j>i$ and $g \in G$ as these are the monomials with more than one reduction using the relations.  Bergman's Diamond Lemma says it is enough to check that the above monomials are reduction-unique.  

\begin{thm}\label{ifftheorem}
The factor algebra $\hat{\SH}_{\q,\kappa,1}$ is a \tqdha\ if and only if \begin{enumerate}[label=(\roman*)]
\item $G$ acts on $\Lambda_{\q}(V)$ by automorphisms and $q_{ij} = q_{ji}^{-1}$ for $i \neq j$,
\item $\kappa$ is a quantum 2-form,
\item For all $h$ in $G$ and $1 \leq i < j < k \leq n$, 
$$0 = (q_{ik}q_{jk}{}^hv_k - v_k) \kappa_h(v_i,v_j) + (q_{jk}v_j - q_{ij}{}^hv_j) \kappa_h(v_i,v_k) + ({}^hv_i - q_{ij}q_{ik}v_i) \kappa_h(v_j,v_k),$$
\item For all $g,h$ in $G$ and all $1 \leq r < s \leq n$,
$$\kappa_{h^{-1}gh}(v_r,v_s) = \sum_{i < j} \textrm{det}_{rsij}(h) \kappa_g(v_i,v_j),$$
\item For all $h$ in $G$ and $1 \leq i < j \leq n$,
$$0 = q_{ij}v_i\kappa_h(v_i,v_j) + {}^hv_i\kappa_h(v_i,v_j),$$
and 
$$0 = q_{ij}{}^hv_j \kappa_h(v_i,v_j) + v_j \kappa_h(v_i,v_j), \textrm{ and}$$ 
\item For all $g, h \in G$ and $i <j$
$$\sum_{i<j} (g_i^rg_j^r) \kappa_h(v_i,v_j) = 0.$$ 
\end{enumerate}
\end{thm}

\begin{proof}
Suppose factor algebra $\hat{\SH}_{\q,\kappa,1}$ is a \tqdha.  Then by Proposition ~\ref{forward} we have that (i) and (ii) are satisfied.  Since $\hat{\SH}_{\q,\kappa,1}$ is an associative algebra, we know that for all $1 \leq i<j<k \leq n$, $0 = v_k(v_jv_i) - (v_kv_j)v_i$.  We reorder the $v_i, v_j, v_k$ in ascending order to recover condition (iii).\begin{align*} 0 =& v_k(v_jv_i) - (v_kv_j)v_i  \\
     =& q_{ij}v_kv_iv_j + v_k \kappa(v_i,v_j) - q_{jk}v_jv_kv_i - \kappa(v_j,v_k)v_i \\
     =& q_{ij}q_{ik}v_iv_kv_j + q_{ij}\kappa(v_i,v_k)v_j + v_k \kappa(v_i,v_j)  - q_{jk}q_{ik}v_jv_iv_k - q_{jk}v_j\kappa(v_i,v_k) -  \kappa(v_j,v_k)v_i \\
     =& q_{ij}q_{ik}v_i(q_{jk}v_jv_k + \kappa(v_j,v_k)) + q_{ij} \kappa(v_i,v_k)v_j +v_k \kappa(v_i,v_j) \\
     &- q_{jk}q_{ik}(q_{ij}v_iv_j + \kappa(v_i,v_j))v_k - q_{jk}v_j\kappa(v_i,v_k) - \kappa(v_j,v_k)v_i \\
     =& q_{ij}q_{ik}q_{jk}v_iv_jv_k + q_{ij}q_{ik}v_i \kappa(v_j,v_k)+ \sum_{g \in G}q_{ij}  \,^gv_j \kappa_g(v_i,v_k)g +  v_k \kappa(v_i,v_j) -q_{jk}q_{ik}q_{ij}v_iv_jv_k \\&-\sum_{g \in G} q_{jk}q_{ik}\,^gv_k \kappa_g(v_i,v_j)g - q_{jk}v_j\kappa(v_i,v_k) - \sum_{g \in G} \,^gv_i \kappa_g(v_j,v_k)g. \end{align*}  

\noindent By rearranging terms, we have that for all $g \in G$,
$0 =   (q_{ij}q_{ik}v_i-\,^gv_i )\kappa_g(v_j,v_k)+ (q_{ij}  \,^gv_j -q_{jk}v_j) \kappa_g(v_i,v_k) +  (v_k- q_{jk}q_{ik}\,^gv_k) \kappa_g(v_i,v_j) $ which is equivalent to condition (iii).  Condition (iv) is the result of considering $0 = h(v_sv_r)h^{-1} - h(q_{rs}v_rv_s+\kappa(v_r,v_s))h^{-1}$. 
If we consider
\begin{align*} 0&= v_j(v_i^2) - (v_jv_i)v_i \\
    & = -q_{ij}v_iv_jv_i - \kappa(v_i,v_j)v_i\\
& = -q_{ij}v_i (q_{ij}v_iv_j + \kappa(v_i,v_j)) - \kappa(v_i,v_j)v_i \\
& = - q_{ij}v_i\kappa(v_i,v_j) - \kappa(v_i,v_j)v_i,
\end{align*} then for all $g \in G$ and all $i < j$, we need $(-^gv_i - q_{ij}v_i)\kappa_g(v_i,v_j) = 0$.  This is the first part of condition (v).  Similarly, $0= (v_j^2)v_i - v_j(v_jv_i)$ gives us the second part of condition (v).  

Finally, we consider 
\begin{align*}
     0 & = g(v_r^2) - (gv_r)v_r  \\
     & = \sum_i g_i^rv_ig v_r\\
     & = (\sum_i g_i^r v_i)( \sum_j g_j^r v_j) g \\
     & = (\sum_{i<j} g_i^rg_j^r v_iv_j 
     + \sum_{j< i}  g_i^rg_j^r v_iv_j) g\\
     & = (\sum_{i<j} g_i^rg_j^r v_iv_j + \sum_{i<j} q_{ij}g_j^r g_i^rv_iv_j + g_j^rg_i^r \kappa(v_ik,v_j) ) g\\
     & = \sum_{i<j}(g_i^rg_j^r + q_{ij}g_j^rg_i^r)v_iv_jg + g_j^rg_i^r \kappa(v_i,v_j)g
\end{align*}
which gives condition (vi).

We have checked all of the monomials required by Bergman's Diamond Lemma to see if they are reduction-unique.  Therefore the factor algebra $\hat{\SH}_{\q,\kappa,1}$ is a truncated quantum Drinfeld Hecke algebra if and only if conditions (i)-(vi) are met.

\end{proof}

We get comparable conditions on $\hat{\SH}_{\q,\kappa,1}$ being a \tqdha\ to the non-truncated case \cite[Theorem 7.6]{LS}.  As it compares to the conditions in \cite{LS}, condition (i) is analogous, conditions (ii)-(iv) are identical, and the added conditions (v)-(vi) correspond to the added truncation.  

We end this Section by providing some insight into the conditions for which $\kappa$ will lead to a truncated quantum Drinfeld Hecke algebra for a fixed set of quantum scalars.  The parameter $\kappa$ defines a linear transformation $\kappa : V \otimes V \rightarrow \KK G$. 
The set of all parameters $\Hom_{\KK}(V \otimes V , \, \KK G)$
is a $\KK$-vector space.  We are interested in the subset for which $\kappa$ defines a truncated quantum Drinfeld Hecke algebra.

\begin{defn}
A parameter $\kappa$ is admissible if it defines a \tqdha \, $\hat{\SH}_{\q, \kappa, t}$.
\end{defn}

That is, $\kappa$ is admissible if $\hat{\SH}_{\q, \kappa, t}$ satisfies the conditions in Theorem ~\ref{ifftheorem}. See \cite{Drinfeld} for the definition in the setting of affine Hecke algebras.  
We call the set $$P_G = \{ \kappa \in \Hom_{\KK}(V \otimes V , \, \KK G)| ~ \kappa \textrm{ is admissible}\} \subset \Hom_{\KK}(V \otimes V , \, \KK G)$$ the parameter space. As $P_G$ is closed under addition and scalar multiplication, $P_G$ is a subspace of $\Hom_{\KK}(V \otimes V , \, \KK G)$. We use the dimension of $P_G$ over $\KK$ to characterize the \tqdha s that arise from specific finite groups with fixed system $\q$ of quantum parameters.

\begin{prop}\label{diagsupport}
In a \tqdha , only group elements that act diagonally on the vector space can support the parameter space. 
\end{prop}

\begin{proof}
Suppose $\trunh$ is a \tqdha \, and $\kappa_g(v_i,v_j) \neq 0$.  Then condition (iii) and condition (v) from Theorem ~\ref{ifftheorem} imply that $g$ is a diagonal action where $g_i^i = -q_{ij}$, $g_j^j = -q_{ji}$, and $g_k^k = q_{ki}q_{kj}$ for all $ k \neq i,j$.
\end{proof}

\begin{cor}\label{diagelement}
The dimension of the parameter space of a \tqdha \, is bounded by $\binom{n}{2}$ where $n$ is the dimension of the vector space. 
\end{cor}

Example ~\ref{fulldimex} gives an example where the bound is met.


\section{Truncated \qdha s\ as deformations}

Truncated quantum Drinfeld Hecke algebras are deformations of the algebra $\Lambda_{\q}(V) \rtimes G$.  In this Section, we give the conditions for which the algebraic description of truncated quantum Drinfeld Hecke algebras coincide with the notion of a formal deformation.

\begin{defn} 
Let $A$ be an algebra over $\KK$.  A deformation of $A$ over $\KK[t]$ is an associative $\KK[t]$ algebra with a deformed multiplication  determined by bilinear maps $\mu_i: A \otimes A \rightarrow A$ where $$a * b=ab +\mu_1(a \otimes b)t+\mu_2( a \otimes b)t^2 + ... + \mu_j(a \otimes b)t^k$$ for all $a, b \in A$ and $ab$ is multiplication in $A.$
\end{defn}

Naidu and Witherspoon showed that quantum Drinfeld Hecke algebras over $\CC[t]$ are deformations of $S_{\q'}(V)$ over $\CC[t]$ with $\deg \mu_i=-2i$ for all $i>0$ in \cite[Theorem 2.2]{NW} (see Section 5.3 for the definition of $S_{\q'}(V)$).  Using a generalization of their proof, which was a rephrasing of the proof of \cite[Theorem 3.2]{W}, we recover a similar result here but with the slight modification that the deformed multiplication must respect the truncation on $\Lambda_{\q}(V)$.

\begin{thm} \label{mu}
The \tqdha s over $\KK[t]$ are the deformations of $\Lambda_{\q}(V) \rtimes G$ over $\KK[t]$ with $\deg \mu_i = -2i$ and $\mu_i(v_j,v_j) = 0$ for all $j \in \{1,2,...,n\}$ and $i>0$.
\end{thm}

\begin{proof}
Assume $\trunh$ is a \tqdha\, over $\KK[t]$.  Then, by \cite[Section 4]{ShepWithPBW}, the associated graded algebra of $\trunh$ is isomorphic to $\Lambda_{\q}(V) \rtimes G[t]$.  We want to show that $\trunh$ is a formal deformation of $\Lambda_{\q}(V) \rtimes G$ which meets the conditions on $\mu_i$.

Note that $\Lambda_{\q}(V) \rtimes G$ has a $\KK$-basis given by $\{v_1^{\alpha_1}v_2^{\alpha_2} \cdots v_n^{\alpha_n}g | \alpha_i \in \{0,1\}\}$.  Thus, because $\trunh$ is a truncated quantum Drinfeld Hecke algebra, all elements in $\trunh$ are a $\KK[t]$-linear combination of terms of the form $v_1^{\alpha_1}v_2^{\alpha_2} \cdots v_n^{\alpha_n}g$ for $\alpha_i \in \{0,1\}$ and $g \in G$. 
 Consider two basis elements $v = v_1^{\alpha_1} \cdots v_n^{\alpha_n}g, w = v_1^{\beta_1} \cdots v_n^{\beta_n}h \in \trunh$.  Denote their product in $\trunh$ by $v \ast w$.  Thus $v \ast w = v_1^{\alpha_1} \cdots v_n^{\alpha_n} {}^g (v_1^{\beta_1} \cdots v_n^{\beta_n}) gh$. 
We obtain
\begin{align*}
v \ast w & = vw + \mu_1(v \otimes w)t + \mu_2(v \otimes w) t^2 + \cdots + \mu_k (v \otimes w) t^k\end{align*}
using defining relations on $\trunh$ to reorder the $v_i$ terms in ascending order.  That is, $vw$ is a linear combination of terms of the form $v_1^{\alpha_1} v_2^{\alpha_2} \cdots v_n^{\alpha_n}gh$ with $\alpha_i \in \{0,1\}$, $\mu_1(v \otimes w)$ is the sum of all terms coming from reordering a single adjacent pair $v_i v_j$, $\mu_2(v \otimes w)$ is the sum of all terms coming from reordering a two adjacent pairs, and similar for higher $t$ powers.  

Because the degree drops by two with every application of the relation, the sum is finite ($k$ is the number of flips of adjacent terms to put the $v_j$'s in ascending order) and $\deg \mu_i = -2i$ for all $i > 0.$  Considering the special case where $v=v_k$ and $w=v_k$ gives $\mu_i(v_k,v_k) = 0$ for all $k \in \{1,2,...,n\}$ and $i>0$.  The operation $\ast$ is $\KK$-linear, making $\mu_i$ linear.  Finally, as $\ast$ is associative, $\trunh$ is a deformation of $\Lambda_{\q}(V) \rtimes G$ over $\CC[t]$. 

On the other hand, assume $D$ is a deformation of $\Lambda_{\q}(V) \rtimes G$ over $\KK[t]$ with the degree of $\mu_i$ equal to $-2i$ and $\mu_i(v_k,v_k) = 0$ for all $k \in \{1,2,...,n\}$ and $i>0$.  We want to show that $D$ is isomorphic to some $\hat{\SH}_{\q,\kappa,t}$.  Because $D$ is a deformation, $D \cong \Lambda_{\q}(V) \rtimes G[t]$ as a vector space over $\KK[t]$. 
Consider the map $\phi: T(V) \rtimes G[t] \rightarrow D$ that sends $$\phi(v_i) = v_i \textrm{ and } \phi(g) = g.$$ We may extend $\phi$, in the first component, to a unique algebra homomorphism from $T(V)$ to $D$.  By the degree condition, we have $\mu_i(\KK G, \KK G) = 0=\mu_i(V, \KK G)=\mu_i(\KK G, V)$.  Therefore
\begin{align*}
\phi((v_{i_1} \cdots v_{i_l}g)(v_{j_1} \cdots v_{j_m}h)) &= v_{i_1} \ast \cdots \ast v_{i_l} \ast g \ast v_{j_1} \ast \cdots \ast v_{j_m} \ast h\\
 &= v_{i_1} \ast \cdots \ast v_{i_l} \ast{}^g v_{j_1} \ast \cdots \ast{}^gv_{j_m} \ast gh \\
 &=\phi(v_{i_1} \cdots v_{i_l}{}^gv_{j_1} \cdots{}^gv_{j_m}gh) 
\end{align*}
and the map can be extended to $\KK[t]$-algebra homomorphism.  The first isomorphism theorem, with this $\phi$, gives us our desired isomorphism between $D$ and $\hat{\SH}_{\q,\kappa,t}$.

We will show that $\phi$ is surjective by strong induction on degree.  By construction, $\phi(g) = g$ and $\phi(v_ig) = v_i \ast g = v_ig + \mu_1(v_i,g)t = v_ig$.  Now, let $v_{i_1} \cdots v_{i_m}g$ be an arbitrary monomial of $D$.  By the induction hypothesis, $v_{i_2} \cdots v_{i_m}g \in \Ima(\phi)$.  Thus $\phi(X) = v_{i_2} \cdots v_{i_m}g$ for some $X \in T(V) \rtimes G[t]$.  Then
\begin{align*}
\phi(v_{i_1}X) & = v_{i_1} \ast \phi(X) \\
 & = v_{i_1} \ast (v_{i_2} \cdots v_{i_m}g) \\
 & = v_{i_1}v_{i_2} \cdots v_{i_m}g + \mu_1(v_{i_1}, v_{i_2} \cdots v_{i_m}g)t + \mu_2(v_{i_1}, v_{i_2} \cdots v_{i_m}g)t^2 + \cdots
\end{align*}
Because $\deg(\mu_i)=2i$, we know $\deg(\mu_j(v_{i_1}, v_{i_2} \cdots v_{i_m}g)t^j)<m+1$ for all $j$ in the finite sum above.  That is, by the induction hypothesis, $\mu_j(v_{i_1}, v_{i_2} \cdots v_{i_m}g)t^j \in \ima \phi$ for all $j$.  Subtracting the appropriate pre-images, we have $v_{i_1}v_{i_2} \cdots v_{i_m}g$ is in the image of $\phi$. 

Now we determine the kernel of $\phi$. It is clear that $v_i^2 \in \ker \phi$ because $\mu_k(v_i \otimes v_i)=0$ for all $k>0$ and $v_i^2=0$. 
Now assume $i \neq j$, then
$$\phi(v_i v_j) = v_i \ast v_j = v_iv_j + \mu_1(v_i,v_j)t \textrm{ and } \phi(v_j v_i) = v_j \ast v_i = v_jv_i + \mu_1(v_j,v_i)t.$$
Note we do not get higher $\mu_i$ terms because of the degree condition.  Also because of the degree conditions, we know $\mu_1(v_i,v_j) \in \KK G$ and $\mu_1(v_j,v_i) \in \KK G$. Using the above expression we get 
\begin{align*}
0=\phi(v_jv_i-q_{ij}v_iv_j) & = v_jv_i + \mu_1(v_j,v_i)t - q_{ij}v_iv_j - q_{ij}\mu_1(v_i,v_j)t \\
 & = v_jv_i - q_{ij}v_iv_j + t(\mu_1(v_j,v_i) - q_{ij}\mu_1(v_i,v_j))\\
 & = 0 + (\mu_1(v_j,v_i) - q_{ij}\mu_1(v_i,v_j))t.
\end{align*}
Because $\phi(g)=g$ for all $g \in G$, 
\begin{equation}
v_jv_i - q_{ij}v_iv_j - \mu_1(v_j,v_i)t + q_{ij}\mu_1(v_i,v_j)t \in \ker \phi . \label{1}
\end{equation}

Let $I[t]$ be the ideal generated by the terms of (\ref{1}) and $v_i^2$.  It is clear that $I[t] \subseteq \ker \phi$. Note $\phi$ induces a surjection $T(V) \rtimes G[t] / I[t] \rightarrow D$.  Because the dimension of $D$ in each degree is less or equal to the corresponding dimension of $T(V) \rtimes G[t] / I[t]$, $I[t]$ is, in fact, equal $\ker \phi$. 

Therefore, by the first isomorphism theorem, $\phi$ induces an isomorphism of algebras $D \cong T(V) \rtimes G[t] / I[t]$ where the right hand side is also isomorphic to an $\hat{\SH}_{\q,\kappa,t}$.  Moreover, the associated graded algebra of $T(V) \rtimes G[t] /I[t]$ is isomorphic to $\Lambda_{\q}(V) \rtimes G[t]$, making $D$ a \tqdha. 
\end{proof}

In order for the deformed algebra of $A$ to be associative, $\mu_1$ must also be a Hochschild 2-cocycle.  That is, for all $a,b,c \in A$, \begin{align*} a\mu_1(b\otimes c)-\mu_1(ab \otimes c) +\mu_1(a \otimes bc) - \mu_1(a \otimes b)c=0.\end{align*}  
As in \cite{NW}, the proof reveals a relation between the functions $\kappa_g$ of Section 2 and Hochschild 2-cocycles $\mu_1$ of this Section and the next.  Namely, \begin{align}\label{kappa}\sum_{g \in G} \kappa_g(v_j, v_i)g = \mu_1(v_i \otimes v_j -q_{ji}v_j \otimes v_i)\end{align} for $i \neq j$ and $\kappa_g(v_i, v_i)=0$. This relation on $\kappa$ also motivates the definition of $\epsilon_{\beta}$ in the construction of our projective resolution in the next Section.  As deformations of an algebra are so intimately tied to the Hochschild cohomology of that algebra, we next look at the Hochschild cohomology of $\Lambda_{\q}(V) \rtimes G$ in order to understand the structure of truncated quantum Drinfeld Hecke algebras from this perspective.


\section{Hochschild cohomology}

We use this Section to develop all the homological information we will need.  

\begin{defn}
Let $A$ be an algebra over a field $\KK$ and let $M$ be an $A$-bimodule.  The Hochschild cohomology of $A$ with coefficients in $M$ is $$\HH^* (A, M) := \Ext^*_{A^e}(A, M)$$ where $A^e=A \otimes A^{op}$ and $A$-bimodules are (left) $A^e$-modules by left and right action in the corresponding tensor factor. 
\end{defn}

When $M=A$, we call $\HH^* (A, A)$ simply the Hochschild cohomology of $A$ and the notation is shortened to $\HH^*(A)$.  The standard choice of projective resolution, on which Hochschild cohomology was originally defined, is the bar resolution.

\begin{defn}
Let $A$ be an algebra over a field $\KK$.  The bar resolution of $A$ is $$\mathbf{B}(A): ...\xrightarrow{\delta_3} A \otimes A \otimes A \otimes A \xrightarrow{\delta_2} A \otimes A \otimes A \xrightarrow{\delta_1} A \otimes A \xrightarrow{\delta_0} A \rightarrow 0$$ where $\delta_m(a_0 \otimes a_1 \otimes ... \otimes a_{m+1}) = \sum_{i=0}^m (-1)^i a_0 \otimes ... \otimes a_i a_{i+1} \otimes ... \otimes a_{m+1}.$
\end{defn}

In the case of skew group algebras, Hochschild cohomology has a particular form given by the isomorphism $$\HH^*(A \rtimes G) \cong \HH^*( A, A \rtimes G)^G$$ if $G$ is a finite group acting on a $\KK$-algebra $A$ and $\ch \KK \nmid |G|$.  See \cite{S} for a proof of this result in the more general setting of Hopf Galois extensions and \cite{L} for the case of Hochschild homology.  Recall from the outset we restricted $G$ to groups such that $\ch \KK \nmid |G|$.  This is one of the settings in which we explicitly need this assumption.  The above isomorphism allows us to compute $\HH^*( \Lambda_{\q}(V) \rtimes G)$ by instead finding the $G$-invariant subspace of $\HH^*( \Lambda_{\q}(V), \Lambda_{\q}(V) \rtimes G) = \bigoplus_{g \in G} \HH^*(\Lambda_{\q}(V), \Lambda_{\q}(V)g)$.  Thus we would like to construct a small resolution of $\Lambda_{\q}(V)$ to compute Hochschild cohomology.  One such small resolution of $\Lambda_{\q}(V)$ was given in \cite{Grimley} which we describe in the next subsection.  

\subsection{Projective resolution of $\Lambda_{\q}(V)$}

The bar resolution can be computationally cumbersome so we define a sub-resolution on which we will compute cohomology.  Let us start by defining a generalization of the generators $\tilde{f}_i^n$ of \cite{BGMS}.  We begin with a more intuitive description.  Define
$$\epsilon_{i_1, i_2, ..., i_n}=\sum_{\alpha \in S(i_1, ..., i_n)} (-q)^{\alpha} \otimes v_{\alpha} \otimes 1$$ where $S(i_1, ..., i_n)$ is the set of multi-set permutations on a set with $i_1$ $1$'s, $i_2$ $2$'s, ..., and $i_n$ $n$'s, $v_{\alpha}$ is the $(i_1+i_2+...+i_n)$-fold tensor product in which the orders of $v_1, v_2, ..., v_n$ are given by the multi-set permutation $\alpha$, and $(-q)^{\alpha}$ is the product of the negative quantum coefficients that occur when permuting the $v_j$'s from ascending order to the arrangement given by $\alpha$.  For example, $$\epsilon_{1,0,0,...0}=1 \otimes v_1 \otimes 1,$$ $$\epsilon_{1,1,0,...,0}=1 \otimes v_1 \otimes v_2 \otimes 1 -q_{21} \otimes v_2 \otimes v_1 \otimes 1,$$ $$\epsilon_{2, 1,0, ..., 0} =1 \otimes v_1 \otimes v_1 \otimes v_2 \otimes 1 - q_{21} \otimes v_1 \otimes v_2 \otimes v_1 \otimes 1 + q_{21}^2 \otimes v_2 \otimes v_1 \otimes v_1 \otimes 1.$$  

To more precisely define $\epsilon_{\beta}$, let's introduce some notation.  Let $\beta \in \mathbb{N}^n$ and $\alpha \in \{0,1\}^n$.  Define $\epsilon_{\beta}=\epsilon_{\beta_1, \beta_2,..., \beta_n}$ using the multi-index notation, define $f_{\beta}$ similarly, and define $v^{\alpha}=v_1^{\alpha_a} v_2^{\alpha_2} \cdots v_n^{\alpha_n}$.  Denote $|\beta|=\beta_1+\beta_2+ \cdots + \beta_n$ and finally, denote by $[j]=(0,...,0,1,0,...,0)$ the standard basis $n$-tuple with one 1 in the $j$th coordinate.

Let $f_{(0,0,...,0)}=1$, $f_{[j]}=v_j $ for all $j \in \{1,2,...,n\}$.  If $\beta_j<0$ for some $j \in \{1,2,...,n\}$, then set
 $f_{\beta}=0$.  For arbitrary $\beta \in \mathbb{N}^n$, $$f_{\beta}=\sum_{j=1}^n \prod_{k>j} (-q_{kj})^{\beta_k} f_{\beta-[j]} \otimes v_j .$$  Intuitively, each term in the sum of $f_{\beta}$ is the consequence of moving an arbitrary $v_j$ to the end of the tensor product and multiplying by the negative of the quantum scalars that would have occurred in $\Lambda_{\q}(V)$ from this movement.  The terms $\epsilon_{\beta}$ are then precisely $1 \otimes f_{\beta} \otimes 1$. Note, with this new notation, we can update the relation (\ref{kappa}) to be $\kappa(v_j,v_i)=\mu_1(\epsilon_{[i]+[j]})$ for all $i\neq j$ further motivating its construction.

Consider the complex, built upon these $\epsilon_{\beta}$ terms, $$\mathbf{C}: ...\xrightarrow{d_{m+1}} \bigoplus_{|\beta|=m} \Lambda_{\q}(V) \e_{\beta} \Lambda_{\q}(V) \xrightarrow{d_{m}} \bigoplus_{|\beta|=m-1} \Lambda_{\q}(V) \e_{\beta} \Lambda_{\q}(V) \xrightarrow{d_{m-1}} ...$$ where 
\begin{align*} d_m(\epsilon_{\beta})=& \sum_{j=1}^n (-1)^{\sum_{l<j} \beta_l} ( \prod_{l<j} q_{jl}^{\beta_l} v_j \epsilon_{\beta-[j]} + (-1)^{\beta_j} \prod_{l >j} q_{lj}^{\beta_l} \epsilon_{\beta-[j]} v_j ).
\end{align*} 
It is simple to check that $d^2=0$.  By \cite[Lemma 3.4]{Grimley}, $(\mathbf{C},d)$ is a subcomplex of the bar resolution $(\mathbf{B}(\Lambda_{\q}(V)), \delta)$.  

\begin{prop}
$(\mathbf{C}, d)$ is a graded projective resolution of $\Lambda_{\q}(V)$ as a $\Lambda_{\q}(V)^e$-module.
\end{prop}

\begin{proof}
In \cite[Section 3]{Grimley}, the first author constructed a graded projective $(\Lambda_{\q}(V))^e$-module resolution of $\Lambda_{\q}(V)$ given by the total complex of the twisted tensor product of $n$ copies of the graded projective $\KK[v_i]/(v_i^2)^e$-module resolution of $\KK[v_i]/(v_i^2)$
$$\mathbf{D_i}: ... \xrightarrow{d^i_3} (\KK[v_i]/(v_i^2))^e\langle 2 \rangle \xrightarrow{d^i_2} (\KK[v_i]/(v_i^2))^e \langle 1 \rangle \xrightarrow{d^i_1} (\KK[v_i]/(v_i^2))^e \xrightarrow{\mu} \KK[v_i]/(v_i^2)  \rightarrow 0$$ where $\mu$ is multiplication, and $d^i_m(1 \otimes 1)=v_i \otimes 1 + (-1)^m 1 \otimes v_i$.  See \cite{Grimley} for the details of this construction.  For the purposes of this article, we note that the argument hinged upon an isomorphism between graded modules in the total complex of the twisted tensor product $(\textrm{Tot}(\mathbf{D_1} \otimes^{t_1} \mathbf{D_2} \otimes^{t_2} \cdots \otimes^{t_{n-1}} \mathbf{D_n})_m$ and $\bigoplus_{|\beta|=m} \Lambda_{\q}(V)^e \langle \beta \rangle$ for $\beta \in \ZZ^n$ given by \cite[Lemma 4.3]{BO}.  The former is a graded projective resolution of $\Lambda_{\q}(V)$ as a $\Lambda_{\q}(V)^e$-module by \cite[Lemma 4.5]{BO} and the latter induced complex we claim is isomorphic to $(\mathbf{C},d)$.  The chain map $\phi: \mathbf{C}_m \rightarrow \bigoplus_{|\beta|=m} \Lambda_{\q}(V)^e \langle \beta \rangle$ given by sending $\e_{\beta}$ to the copy of $1 \otimes 1$ with homological degree $\beta$, induces the graded isomorphism we need between $(\mathbf{C}, d)$ and the latter complex. 
\end{proof}

\begin{defn} A projective resolution is $G$-compatible if $G$ acts on each algebra in the resolution and the action commutes with the differentials.
\end{defn}

The bar resolution, $\mathbf{B}(\Lambda_{\q}(V))$, is $G$-compatible for all group actions on $\Lambda_{\q}(V)$.  It turns out that this is all we need for the subresolution $\mathbf{C}$ to be $G$-compatible as well.

\begin{lemma}\label{commute}
The group action on $V$ extends to a $G$-compatible action on $\mathbf{C}$ if and only if 
\begin{center}
$(1+q_{ij})(g_i^r g_j^s-q_{sr}g_i^s g_j^r)=0$ and \\
$(1+q_{ij})g_i^r g_j^r=0$ \\
\end{center} for all $r \neq s$ and $i \neq j$.
\end{lemma}

Comparing to Lemma \ref{LambdaAut}, $\mathbf{C}$ is $G$-compatible precisely when $G$ acts by automorphisms on $\Lambda_{\q}(V)$.  To close this section, let's record how elements of $G$ must act on the generators of the resolution, $\epsilon_{\beta}$.  Lemma \ref{commute} implies $$^g \epsilon_{[r]+[s]} = \sum_{l \leq k} (g_l^r g_k^s - q_{sr} g_l^s g_k^r) \epsilon_{[l]+[k]} \textrm{ and } ^g \epsilon_{2[r]} = \sum_{l \leq k} g_l^r g_k^r \epsilon_{[l]+[k]}$$ for all $g \in G$ and $r<s$.

\subsection{Relevant 2-cocycles}
To compute $\HH^*(\Lambda_{\q}(V), \Lambda_{\q}(V)g)$, we must apply the functor \newline $\Hom_{\Lambda_{\q}(V)^e}(-, \Lambda_{\q}(V)g)$ to our projective resolution from the previous section.  

Let $\epsilon_{\beta}^*$ be the dual $\Lambda_{\q}(V)^e$-module homomorphism, sending $\epsilon_{\beta}$ to $1$ and all other terms to $0$.  After passing through the isomorphism $\Hom_{(\Lambda_{\q}(V))^e}(\Lambda_{\q}(V) \e_{\beta} \Lambda_{\q}(V), \Lambda_{\q}(V)g) \cong \Hom_{\KK}(\e_{\beta}, \Lambda_{\q}(V)g)$, 
on which homomorphisms are completely determined by the image of $\e_{\beta}$, we get the complex

$$\Hom_{(\Lambda_{\q}(V))^e}(\mathbf{C}, \Lambda_{\q}(V) g): ...\xrightarrow{d^{m-1}} \bigoplus_{|\beta|=m-1} \Lambda_{\q}(V) g \e_{\beta}^* \xrightarrow{d^{m}} \bigoplus_{|\beta|=m} \Lambda_{\q}(V)g \e_{\beta}^* \xrightarrow{d^{m+1}} ...$$ where 
\begin{align*} d^m((v^{\alpha} g)\epsilon_{\beta}^*)=&\sum_{j=1}^n (-1)^{\sum_{l<j} \beta_l} (\prod_{l<j} q_{jl}^{\beta_l}(v_j v^{\alpha} g) - (-1)^{\beta_j} \prod_{l >j} q_{lj}^{\beta_l} (v^{\alpha} (^g v_j)  g)  ) \epsilon_{\beta+[j]}^*\\
=&\sum_{j=1}^n (-1)^{\sum_{l<j} \beta_l} ( \prod_{l<j}q_{jl}^{\beta_l- \alpha_l} (v^{\alpha+[j]}  g) - (-1)^{\beta_j} \prod_{l >j} q_{lj}^{\beta_l} \sum_{i=1}^n g_i^j \prod_{i<l} q_{il}^{\alpha_l}(v^{\alpha + [i]}  g)  ) \epsilon_{\beta+[j]}^*.
\end{align*}

For the purposes of recovering the structure of \tqdha s, we are interested in those 2-cocycles that 
 meet the degree requirement of Theorem \ref{mu}.  That is, we want to find $G$-invariant elements $\eta \in \Hom_{(\Lambda_{\q}(V))^e}(\mathbf{C}_2, \Lambda_{\q}(V)g)$ such that $d^3(\eta)=0$ and $\deg(\eta)=-2$.  All elements in $\mathbf{C}_2$ have homological and polynomial degree 2 therefore the image of $\eta$ must be in $\KK g$.  We will call such maps constant maps.

Let $\eta$ be an arbitrary constant 2-cocycle.  Then $\eta$ is of the form $$\eta = \sum_{g \in G} \sum_{1 \leq r \leq s \leq n} (\kappa_{rs}^g  g) \epsilon_{[r]+[s]}^*$$ for some scalars $\kappa_{rs}^g \in \KK$.  Note $\epsilon_{[r]+[s]}^*$ is order independent but $\kappa_{rs}^g$ is only defined for $r \leq s$.  $(\kappa_{rs}^g g) \epsilon_{[r]+[s]}^*$ is the homomorphism sending $$1 \otimes v_r \otimes v_s \otimes 1 -q_{sr}\otimes v_s \otimes v_r \otimes 1 \mapsto \kappa_{rs}^g g.$$  We may extend 
$\{\kappa_{rs}^g\}_{r<s}$ to all $r \neq s$ by defining $\kappa_{sr}^g$ to be $g$-component of the image of $-q_{sr}\epsilon_{[r]+[s]}=1 \otimes v_s \otimes v_r \otimes 1 -q_{rs} \otimes v_s \otimes v_r \otimes 1$ under $\eta$.  With this 
compatible extension, $\kappa_{sr}^g=-q_{rs} \kappa_{rs}^g$.

The necessary and sufficient conditions for which the coefficient of $\epsilon_{[i]+[j]+[k]}$ is 0 in $d^3(\eta)$ for any $i \leq j \leq k$ in $\HH(\Lambda_{\q}(V), \Lambda_{\q}(V) g)$ are 
\begin{align}\label{(5.4)} \kappa_{jk}^g(v_i -&q_{ji}q_{ki}(^g v_i )) - \kappa_{ik}^g(q_{ji}v_j -q_{kj}{}^g v_j ) + \kappa_{ij}^g(q_{ki}q_{kj}v_k-{}^g v_k) = 0, \end{align} \begin{align}\label{(5.5)} \kappa_{jj}^g(q_{kj}^2v_k -&{}^g v_k ) + \kappa_{jk}^g(v_j +q_{kj}{}^g v_j ) = 0, \end{align} \begin{align}\label{(5.6)} \kappa_{jj}^g(v_i -&q_{ji}^2 {}^g v_i ) - \kappa_{ij}^g(q_{ji}v_j +{}^g v_j) = 0, \textrm{ and} \end{align} \begin{align}\label{(5.7)} \kappa_{kk}^g(v_k -{}^g v_k) = 0. \end{align}

What remains is to determine the $G$-invariant subspace, placing additional constraints on our 2-cocycles.  Let $h \in G$ and $r<s$, then

\begin{align*}
(^h \eta)(\epsilon_{[r]+[s]}) =& h \sum_{g \in G} \sum_{i<j} (\kappa_{ij}^g  g) \epsilon_{[i]+[j]}^*(^{h^{-1}} \epsilon_{[r]+[s]}) h^{-1} \\
=& h \sum_{g \in G} \sum_{i<j} (\kappa_{ij}^g g) \epsilon_{[i]+[j]}^*(\sum_{l \leq k} [(h^{-1})_l^r (h^{-1})_k^s - q_{sr} (h^{-1})_l^s (h^{-1})_k^r] \epsilon_{[k]+[l]}) h^{-1} \\
=& h \sum_{g \in G} \sum_{i<j} (\kappa_{ij}^g  g) [(h^{-1})_i^r (h^{-1})_j^s - q_{sr} (h^{-1})_i^s (h^{-1})_j^r] h^{-1} \\
=& \sum_{g \in G} \sum_{i<j} \kappa_{ij}^g [(h^{-1})_i^r (h^{-1})_j^s - q_{sr} (h^{-1})_i^s (h^{-1})_j^r]  hgh^{-1} 
\end{align*}

and 
\begin{align*}
(^h \eta)(\epsilon_{2[r]}) =&  h \sum_{g \in G} \sum_{i<j} (\kappa_{ij}^g  g) \epsilon_{[i]+[j]}^*(^{h^{-1}} \epsilon_{2[r]}) h^{-1} \\
=&  h \sum_{g \in G} \sum_{i<j} (\kappa_{ij}^g g) \epsilon_{[i]+[j]}^*(\sum_{l \leq k} (h^{-1})_l^r (h^{-1})_k^r  \epsilon_{[k]+[l]})h^{-1}\\
=& \sum_{g \in G} \sum_{i<j} \kappa_{ij}^g (h^{-1})_i^r (h^{-1})_j^r   hgh^{-1}. 
\end{align*}
In order for $^h \eta = \eta$, we must have,  for all $h,g \in G$ and $r<s$, \begin{align}\label{(5.8)} \sum_{i<j} \kappa_{ij}^g [(h^{-1})_i^r (h^{-1})_j^s - q_{sr} (h^{-1})_i^s (h^{-1})_j^r] = \kappa_{rs}^{h g h^{-1}} \textrm{ and } \sum_{i<j} (h^{-1})_i^r (h^{-1})_j^r \kappa_{ij}^g = 0. \end{align}

Now that we have the necessary information for determining Hochschild 2-cocycles of \newline $\Lambda_{\q}(V) \rtimes G$, we can characterize those that are truncated quantum Drinfeld Hecke algebras.

\begin{thm}
If the $G$ action on $V$ extends to an action on $\Lambda_{\q}(V)$, then each constant Hochschild 2-cocycle of $\Lambda_{\q}(V) \rtimes G$ that sends $\epsilon_{2[i]} \mapsto 0$ for all $i \in \{1, 2, ..., n\}$ produces a truncated quantum Drinfeld Hecke algebra.
\end{thm}

\begin{proof}
Let $$\eta=\sum_{g \in G} \sum_{1 \leq r \leq s \leq n} (\kappa_{rs}^g  g) \epsilon_{[r]+[s]}^*$$ be a constant 2-cocycle.

We compare the conditions of 2-cocycles from this section to the conditions on truncated quantum Drinfeld Hecke algebras given in Theorem \ref{ifftheorem}.  To compare the two sections, we set $\kappa_{ij}^g=\kappa_g(v_j,v_i)$, as suggested by the proof of Theorem \ref{mu}.  Theorem \ref{ifftheorem} condition (i) holds as a result of Lemma \ref{commute}.  The requirement that cocycles send $\epsilon_{2[i]} \mapsto 0$ for all $i \in \{1, 2, ..., n\}$, forces $\kappa_{ii}^g=0$ $\forall i \in \{1,2,...,n\}$ by definition.  Therefore $\kappa$ is a quantum 2-form and condition (ii) is met.  

Because we set $\kappa_{ij}^g=\kappa_g(v_j,v_i)$, to compare to Theorem \ref{ifftheorem}, we should first substitute $\kappa_{ij}^g$ with $-q_{ji} \kappa_{ji}^g$ in all cases containing such terms for $i \neq j$.  We describe all other manipulations necessary to match the conditions in Theorem \ref{ifftheorem} precisely below.

As a result of $\kappa_{ii}^g$ being 0 for all $i \in \{1,2,...,n\}$, equations \ref{(5.5)} and \ref{(5.6)} become $$\kappa_{jk}^g(v_j+q_{kj} {}^gv_j)=0 \textrm{ and } \kappa_{ij}^g (q_{ji}v_j+ {}^gv_j)=0$$ respectively for $i<j<k$.  Multiply the first equation by $-q_{jk}^2$ and the second equation by $-q_{ij}^2$ to match the first and second conditions of (v).

If we multiply equation (\ref{(5.4)}) by $q_{ij}q_{ik}q_{jk}$, then (\ref{(5.4)}) becomes condition (iii).  Condition (iv) is the first condition in equation (\ref{(5.8)}), after rewriting the scalar using Lemma \ref{commute}. The first condition of (vi) is met by the second condition in equation (\ref{(5.8)}), after using the relation $-q_{ij}g_i^r g_j^r=g_i^r g_j^r$ from Lemma \ref{commute} to rewrite the scalar.  
\end{proof}


\section{Basic examples}

We conclude this paper by providing examples that show the range of the class of truncated quantum Drinfeld Hecke algebras.

\subsection{Diagonal actions}

We begin by comparing the PBW and Hochschild cohomological conditions for group actions that are strictly diagonal.  In this case, Hochschild cohomology is known.

\begin{thm}[\cite{Grimley}, Theorem 3.3]\label{gendiag}
Assume $G$ acts diagonally on $\Lambda_{\q}(V)$ by $^g v_i= g_i^i v_i$ for some $g^i_i \in \KK^*$, then $\HH^m(\Lambda_{\q}(V), \Lambda_{\q}(V) \rtimes G)$ is isomorphic to $$\bigoplus_{g \in G} \bigoplus_{\substack{\beta \in \mathbb{N}^n \\ |\beta|=m}} \bigoplus_{\substack{\alpha \in \{0,1\}^n \\ \beta-\alpha \in C_g}} \Span_{\KK} \{(v^{\alpha} g) \epsilon_{\beta}^*\}$$ where $C_g = \{\gamma \in (\Z^{\geq -1})^n | \forall~ i, \gamma_i=-1 \textrm{ or } (-1)^{\gamma_i}\prod_{j \neq i} q_{ij}^{\gamma_j} = g_i^i \}$. Moreover, $\HH^m(\Lambda_{\q}(V) \rtimes G)$ is isomorphic to the $G$-invariant subspace of $\HH^m(\Lambda_{\q}(V), \Lambda_{\q}(V) \rtimes G)$.
\end{thm}

The curious reader should compare this theorem to the results of Naidu, Shroff, and Witherspoon \cite[Theorem 4.1]{NSW} in the non-truncated setting. 

With this result in mind, we can systematically determine the $n$-tuples $\gamma=\beta-\alpha$ satisfying $\gamma_i=-1 \textrm{ or } (-1)^{\gamma_i}\prod_{j \neq i} q_{ij}^{\gamma_j} = g_i^i$ for all $i \in \{1,2,...,n\}$ such that $|\beta|=2$.  We can simplify this expression further to isolate those which result in a truncated quantum Drinfeld Hecke algebra.

\begin{cor}\label{diag}
The subspace of $\HH^m(\Lambda_{\q}(V), \Lambda_{\q}(V) \rtimes G)$ consisting of constant 2-cocycles which send $v_i \otimes v_i \mapsto 0$ is isomorphic to $$\bigoplus_{g \in G} \bigoplus_{\substack{i \textrm{ s.t.} \\ \forall j \neq i \\ -q_{ij}=g_i^i}} \Span_{\KK} \{(g) \epsilon_{[i]+[j]}^*\} \bigoplus_{\substack{i \textrm{ s.t.} \\ \forall j,k \neq i \\ q_{ij}q_{ik}=g_i^i}} \Span_{\KK} \{(g) \epsilon_{[j]+[k]}^*\}.$$
\end{cor}

Using Corollary \ref{diag}, we can check fewer relations to generate truncated quantum Drinfeld Hecke algebras.  In the examples below, we include the full description of Hochschild 2-cohomology, given by Theorem \ref{gendiag}, for a more complete description of possible deformations.  Recall that the translation between Hochschild cohomology and truncated quantum Drinfeld Hecke algebras is given by $\sum_{g \in G} \kappa_g(v_j, v_i)g = \mu_1(\epsilon_{[i]+[j]})$ for $\mu_1$ a 2-cocycle and $i<j$.  

We start with an example of a \tqdha \, where the bound on the dimension of the parameter space is met.  See Section 2 for details on the parameter space.

\begin{example}\label{fulldimex}
Let $G = \left\{ \left(\begin{smallmatrix} -1 & 0 & 0 \\ 0& -1 & 0 \\ 0 & 0 & -1 \end{smallmatrix}\right), \left(\begin{smallmatrix} 1 & 0 & 0 \\ 0& 1 & 0 \\ 0 & 0 & 1 \end{smallmatrix}\right) \right\} = \{ g, I\}$. Let $q_{12} = -1$ and $q_{13}= q_{23}=1$. Then $\hat{\SH}_{\q,\kappa,1}$ is a \tqdha \, where $\hat{\SH}_{\q,\kappa,1}$ is generated by $v_1, v_2, v_3$ and $h \in G$ with relations 
\begin{align*}
    v_2v_1 & = -v_1v_2 + m_1 I,\\
    v_3v_1 &= v_1v_3 + m_2 g,\\
    v_3v_2 &= v_2v_3 + m_3 g,\\
    v_i^2 & = 0 \text{ for $i=1,2,3$ }
\end{align*}
where $m_1, m_2, m_3 \in \KK$ and the dimension of the parameter space is $3 = \binom{3}{2}$.
Compare the above to the $G$-invariant subspace of 2-cocycles for this example which is $$\Span_{\KK}\{(I) \epsilon_{1,1,0}^* , (g) \epsilon_{1,0,1}^* , (g)\epsilon_{0,1,1}^* \}.$$
\end{example}

We include a non-truncated example of a quantum Drinfeld Hecke algebra, given in \cite{Uhl}, to contrast with the truncated example immediately following it.

\begin{example}
\label{Ex3}
Let $g=\left(\begin{smallmatrix} -\omega^2 & 0 & 0 \\ 0 & -\omega & 0 \\ 0 & 0 & 1 \end{smallmatrix}\right)$, $G = \left\langle g \right\rangle$ and $q_{12}=-1, q_{23} =\omega = q_{31}$ where $\omega = e^{\frac{2 \pi i}{3}}$ and $V = \CC^3$. Then $\SH_{\q,\kappa}$ is a \qdha \, where $\SH_{\q,\kappa}$ is generated by $v_1,v_2,v_3$ and $h \in G$ with relations
\begin{align*}
v_2v_1 &= -v_1v_2 + m_0I + m_1g + m_2{g^2} + m_3{g^3} + m_4{g^4} +m_5{g^5}, \\
v_3v_2 &= \omega v_2v_3, \\ 
v_3v_1 &= \omega^2 v_1v_3, 
\end{align*}
and $m_i \in \CC$. Here we have that $\kappa_h(v_1,v_2)$ is arbitrary for all $h \in G$ and that $\dim(P_G) = 6 = |G|$.
\end{example}

Note that $\kappa_g(v_1,v_2) \neq 0$, but $g_1^1 \neq -q_{12}$ and $g_2^2 \neq -q_{21}$ as would be required in the truncated version.

\begin{example}
As in the above example, let $g= \left(\begin{smallmatrix} -\omega^2 & 0 & 0 \\ 0 & -\omega & 0 \\ 0 & 0 & 1 \end{smallmatrix}\right)$, $G = \left\langle g \right\rangle$ and quantum scalars $q_{12}=-1, q_{23} =\omega = q_{31}$ where $\omega = e^{\frac{2 \pi i}{3}}$ and $V = \CC^3$.  In the truncated setting, we have $\hat{\SH}_{\q,\kappa,1}$ is a \tqdha \, where $\hat{\SH}_{\q,\kappa,1}$ is generated by $v_1,v_2,v_3$ and $ h\in G$ with relations 
\begin{align*}
    v_2v_1 &= -v_1v_2 + m_0I,\\
     v_3v_2 & = \omega v_2v_3,\\
    v_3v_1 & = \omega^2 v_1v_3, \\
     v_i^2 & = 0 \text{ for $i=1,2,3$ }
\end{align*}
where $m_0 \in \CC$.  The dimension of the parameter space is 1.
Compare this to the $G$-invariant subspace of 2-cocycles for this example which is $$\Span_{\KK}\{(v_1v_2 g^2)\epsilon_{0,0,2}^*, (g^2) \epsilon_{0,0,2}^*, (v_1v_2 g^3)\epsilon_{0,0,2}^*, (v_2 g^3)\epsilon_{2,0,0}^*,$$
$$(v_1 g^3)\epsilon_{0,2,0}^*, (v_1v_2 g^4)\epsilon_{0,0,2}^*, (v_1v_2 I)\epsilon_{0,0,2}^*, (v_1v_2 I)\epsilon_{1,1,0}^*,$$ $$(v_1v_3 I)\epsilon_{1,0,1}^*, (v_2v_3 I)\epsilon_{0,1,1}^*, (I)\epsilon_{1,1,0}^*,\textrm{ and } (v_1v_2 g^5)\epsilon_{0,0,2}^*\}.$$  Of these, the only \textit{constant} generators are $(I)\epsilon_{1,1,0}^*$ and $(g^2)\epsilon_{0,0,2}^*$.  However, including the condition that cocycles must send $v_1 \otimes v_i \rightarrow 0$ excludes the latter cocycle.  
\end{example}

Note the Hochschild cohomology in this case gave rise to other deformations that are not truncated quantum Drinfeld Hecke algebras.

Comparing this example to the previous, we see that in the non-truncated setting there is some extra freedom in the dimension of the parameter space. Specifically, the dimension of the parameter space in the non-truncated setting is not bound by $\binom{n}{2}$.  However, in the truncated setting there is more freedom on the groups that we can consider as shown in Example ~\ref{lambdaautex}.  

We conclude the diagonal action section with an example with a nontrivial parameter on a non-identity group element. 

\begin{example}
Let $g=\left(\begin{smallmatrix} \omega & 0 & 0 \\ 0 & \omega^2 & 0 \\ 0 & 0 & 1 \end{smallmatrix}\right)$, $G = \left\langle g  \right\rangle$, and quantum scalars $q_{12}=-\omega = q_{23}= q_{31}$ where $\omega = e^{\frac{2 \pi i}{3}}$ and $V = \CC^3$.  $\hat{\SH}_{\q,\kappa,1}$ is a \tqdha \, where $\hat{\SH}_{\q,\kappa,1}$ is generated by $v_1,v_2,v_3$ and $h \in G$ with relations 
\begin{align*}
    v_2v_1 &= -\omega v_1v_2 + m_0g,\\
    v_3v_1 & = -\omega^2 v_1v_3,\\
    v_3v_2 & = -\omega v_2v_3,\\
     v_i^2 & = 0 \text{ for $i=1,2,3$ }
\end{align*}
where $m_0 \in \CC$. The dimension of the parameter space is 1. Note that when attempting the homological computation, we get the system of equations for $g$-acting $$(-1)^{\gamma_1}(-\omega)^{\gamma_2}(-\omega^2)^{\gamma_3}=\omega$$
$$(-1)^{\gamma_2}(-\omega^2)^{\gamma_1}(-\omega)^{\gamma_3}=\omega^2$$
$$(-1)^{\gamma_3}(-\omega)^{\gamma_1}(-\omega^2)^{\gamma_2}=1$$ where $\gamma$ is the difference in homological and polynomial degree.  This gives us the eligible representatives are 
$$(v_1v_2 g)\epsilon_{0,0,2}^*, (g)\epsilon_{1,1,0}^*, (g)\epsilon_{0,0,2}^*, (v_1 g^2)\epsilon_{0,2,0}^*$$ 
$$(v_1 v_2 I)\epsilon_{1,1,0}^*, (v_1v_3 I)\epsilon_{1,0,1}^*, \textrm{ and }(v_2 v_3 I)\epsilon_{0,1,1}^*.$$ The \textit{constant} terms are $(g)\epsilon_{1,1,0}^*, (g)\epsilon_{0,0,2}^*$ and the only one for which $v_i \otimes v_i \mapsto 0$ is $(g)\epsilon_{1,1,0}^*$ which again agrees with the conditions above.
\end{example}

\subsection{Complex reflection groups}
We consider the infinite family of complex reflection groups, $G(r,p,n)$ described by G.C. Shephard and J.A. Todd ~\cite{shephard_finite_1954} to compare with the results of Naidu and Witherspoon ~\cite{NW} who consider them in the non-truncated setting. The group $G(r,p,n)$ is the finite group of $n \times n$ monomial matrices, whose nonzero entries are $r^{th}$ roots of unity and the product of nonzero entries is a $\frac{r}{p}$ root of unity for a fixed $r,p,n \in \ZZ$ where $p \mid r$. 

\begin{example}[Symmetric Group]
\label{sym3}
Consider $\mathfrak{S}_3$ acting on $V=\CC^3$ and $q_{12}=q_{13}=q_{23} = -1$. Then $\hat{\SH}_{\q,\kappa,1}$ is a \tqdha \, generated by $v_1,v_2,v_3, $ and $h \in \mathfrak{S}_3$ with relations 
\begin{align*}
v_2v_1 &= -v_1v_2 + mI , \\
v_3v_1 &= -v_1v_3 + mI , \\
v_3v_2 &= -v_2v_3 + mI, \\
 v_i^2 & = 0 \text{ for $i=1,2,3$ }
\end{align*}
where $m \in \CC$ and $I$ is the identity of the group,
is a \tqdha.  The dimension of the parameter space is 1.

\end{example}

\begin{thm}
Let $\q = \{-1\}$.  For the infinite family of complex reflection groups $G(r,p,n)$, the only nontrivial \tqdha s are $G(1,1,n) = \mathfrak{S}_n$ whose parameter space has dimension 1 and $G(2,2,2)$ whose parameter space also has dimension 1.
\end{thm}

\begin{proof}
Proposition ~\ref{diagsupport} focuses our attention to group elements with diagonal action.  When $r>1$ and $n \neq 2$, it is easy to see that condition (iv) of Theorem ~\ref{ifftheorem} forces $\kappa \equiv 0$.  For $G(2,2,2)$, the identity of the group supports the parameter space.
\end{proof}

\begin{thm}
Let $\q = \{1\}$.  For the infinite family of complex reflection groups $G(r,p,n)$ there are no nontrivial \tqdha s.
\end{thm}

\begin{proof}
 For $n>2$, the only group elements that could support the parameter space would be elements that act diagonally on the vector space, sending two vectors to their negatives and the rest to themselves.  The presence of the permutation $(1 \ 2 \ 3) \in G(r,p,n)$ forces $\kappa \equiv 0$. (Conjugating the group element by $(1 \ 2 \ 3)$ results in another group element with diagonal action that commutes with original element and that in addition to condition (iv) from Theorem  ~\ref{ifftheorem} forces the dimension of the parameter space to be zero.)
 For $n=2$, either there is not an element that satisfies Corollary ~\ref{diagelement} or there is, but condition (iv) is violated by an element that commutes with that element.
\end{proof}

\subsection{A non-truncated \tqdha}

As shown in Section 3, \tqdha s only need the group to act on $\Lambda_{\q}(V)$ by automorphisms and not $S_{\q'}(V)$ as is needed in the non-truncated case.  To compare actions on the two algebras, we must now restrict $\q$ such that $q_{ii}=1$ for all $i \in \{1,2,...,n\}$.  We define $$S_{\q'}(V)=\KK \langle v_1, v_2, ..., v_n| v_j v_i=q_{ij} v_i v_j \textrm{ for } i,j \in \{1, 2, ..., n\} \rangle.$$  Note $S_{\q'}(V)$ is the quantum polynomial algebra $S_{-\q}(V)$, associated with the quantum scalar set $-\q$ and consistent with our definition of $\Lambda_{\q}(V)$.

We now provide an example of a group acting on $\Lambda_{\q}(V)$ by automorphism, which does \textit{not} act on $S_{\q'}(V)$ by automorphisms.  Thus, by considering the truncated case, we actually have some extra freedom that we do not have in the non-truncated case.  This also gives an example of a nonmonomial group.

\begin{example}\label{lambdaautex}
Let $g=\left(\begin{smallmatrix} \sqrt{1-\eta^3} & \eta & 0 \\ \eta^2 & -\sqrt{1-\eta^3} & 0 \\ 0 & 0 & 1 \end{smallmatrix}\right)$, $G = \left\langle g  \right\rangle$, and quantum scalars $q_{ij}=-1$, for all $i \neq j$ where $\eta = e^{\frac{2 \pi i}{5}}$ and $V = \CC^3$ in the truncated setting, we have $\hat{\SH}_{\q,\kappa,1}$ is a \tqdha \, where $\hat{\SH}_{\q,\kappa,1}$ is generated by $v_1,v_2,v_3$ and $h \in G$ with relations 
\begin{align*}
    v_2v_1 &= -v_1v_2 ,\\
    v_3v_1 & = - v_1v_3 + mI,\\
    v_3v_2 & = - v_2v_3 + \eta^3(1-\sqrt{1-\eta^3})mI,\\ 
    v_i^2 & = 0 \text{ for $i=1,2,3$ }
\end{align*}
where $m \in \CC$ and $I$ is the identity of the group. The dimension of the parameter space is 1. 

\end{example}
Note that $G$ does not act on $S_{\q'}(V)$ by automorphisms as $q_{ij} = -1$ for all $i \neq j$ and $G$ is not a monomial matrix group as required by ~\cite[Theorem 11.6]{LS}.  However, by changing the quantum scalars to $q_{12}=1$ and $q_{23} = -1 = q_{31}$, this group would act on $S_{\q'}(V)$ by automorphisms and also produces a nontrivial \qdha \, in that setting (see \cite{Uhl}). 



\end{document}